\documentclass[11pt,reqno]{article}
\usepackage{amsmath,amsfonts,amsthm,amssymb}
\usepackage{graphics}
\usepackage[all]{xy}
\usepackage{latexsym}
\usepackage{multirow}

\theoremstyle{definition}
\newtheorem{thm}{Theorem}[section]
\newtheorem{prop}[thm]{Proposition}
\newtheorem{coro}[thm]{Corollary}
\newtheorem{lem}[thm]{Lemma}
\newtheorem{rem}[thm]{Remark}

\newtheorem{defn}[thm]{Definition}
\newtheorem{exmp}[thm]{Example}

\xyoption{knot}
\UseComputerModernTips \knottips{FF}
\def\Bracket#1{\mathord{\Big\langle\,~ \raise8pt\xybox{0;/r1.3pc/:#1}\,
~\Big\rangle}}

\def\nnewBracket#1{\mathord{\Bigl[\,~ \raise8pt\xybox{0;/r1.3pc/:#1}\,
\Bigr]}}

\def\doubleBracket#1{\mathord{\Bigl[\Bigl[\,~~ \raise8pt\xybox{0;/r1.3pc/:#1}\,
~~\Bigr]\Bigr]}}

\def\tBracket#1{\mathord{\Big(\,~\raise8pt\xybox{0;/r1.3pc/:#1}\,
\Big)}} 

\def\eBracket#1{\mathord{\,~\raise8pt\xybox{0;/r1.1pc/:#1}}}

\def\Fig#1{\mathord{\,~ \raise-10pt\xybox{0;/r0.15pc/:#1}\,
~}}

\begin{document}


\title{Applying Lipson's state models to marked graph diagrams of surface-links}

\author{Yewon Joung
\\{\it\small Department of Mathematics, Graduate School of Natural Sciences,}\\ {\it\small Pusan National University, Busan 609-735, Korea}
\\{\it\small yewon112@pusan.ac.kr}\\\\
Seiichi Kamada
\\{\it\small Department of Mathematics, Osaka City University,}\\
{\it Osaka 558-8585, Japan}\\{\it\small skamada@sci.osaka-cu.ac.jp}\\\\
Sang Youl Lee
\\{\it\small Department of Mathematics, Pusan National University,}\\
{\it Busan 609-735, Korea}\\{\it\small sangyoul@pusan.ac.kr}}

\maketitle

\begin{abstract}
A. S. Lipson constructed two state models yielding the same classical link invariant obtained from the Kauffman polynomial $F(a,z)$. In this paper, we apply Lipson's state models to marked graph diagrams of surface-links, and observe when they induce surface-link invariants.
\end{abstract}



\section{Introduction}

A marked graph diagram is a link diagram possibly with some $4$-valent vertices equipped with markers.
S. J. Lomonaco, Jr. \cite{Lo} and K. Yoshikawa \cite{Yo} introduced a
method of describing surface-links by marked graph diagrams.  Yoshikawa \cite{Yo}
studied surface-links via such diagrams and made a table of surface-links with ``ch-index''  ten or less.
M. Soma \cite{So} studied surface-links described by marked graph diagrams of square-type, and constructed some interesting series of surface-links.  Local moves on marked graph diagrams introduced in Yoshikawa's paper  \cite{Yo} are so-called {\it Yoshikawa moves}.  It is known that two marked graph diagrams present equivalent surface-links if and only if they are related by a finite sequence of Yoshikawa moves \cite{KK, Sw}.  So one can use marked graph diagrams in order to define surface-link invariants.  The third author gave a framework to construct  invariants of surface-links from classical link invariants via marked graph diagrams in \cite{Le2, Le1}. Especially he considered invariants derived from a skein relation in \cite{Le4}. The first author, the third author and J. Kim \cite{JKL} defined ideal coset invariants for surface-links.

In this paper, we review Lipson's state-sum invariants $R_D$ and $S_D$ of classical links from \cite{ASL}. In classical case, he defined a $[\cdot]$-state of $D$ to be a labelling of each connected component of $\widehat{D} := D \setminus \{ \mbox{crossings} \}$
with either $1$ or $2$. A $[\cdot]$-state is {\it legal} if at each crossing of the diagram, each label occurs an even number of times.
For each legal state $\sigma$,
he assigned the value $\nu(c, \sigma)$ to each crossing $c$ as shown in Fig.~\ref{fig-legal}, and defined
$[D]$ by
$$[D] = [D](x,y,z,w) = \sum_{\text{legal states}} \prod_{\text{crossings}} \nu(c, \sigma).$$
He found two conditions on $x, y, z, w$ for the state-sum $[D]$ to be invariant under Reidemeister moves $\Omega_2$ and $\Omega_3$ (see Fig.~\ref{fig-moves-type-I}).
After normalizing $[D]$, he had the state-sum invariants $R_D$ and $S_D$.
We generalize his state models to marked graph diagrams and investigate conditions so that the state-sum can be used for invariants of surface-links.  Then we obtain two invariants $R'_D$ and $S'_D$ of surface-links.

We also consider another state-sum $Q_D$ that is  invariant under all Yoshikawa moves except the moves $\Omega_4$ and $\Omega_4'$ (see Fig.~\ref{fig-moves-type-I}).  It gives an obstruction for Yoshikawa moves $\Omega_4$ and $\Omega_4'$.

This paper is organized as follows. In Section 2, we prepare some preliminaries about marked graph diagrams of surface-links and Yoshikawa moves. In Section 3, we review Lipson's state-sum invariants $R_D$ and $S_D$ of classical links. In Section 4, we first define the state-sum $[D]$ of a marked graph diagram $D$ and then generalize  Lipson's state-sum invariants to invariants for surface-links in $\mathbb R^4$, denoted by
$R'_D$ and $S'_D$. In Section 5, we study these invariants. In Section 6, we show that the state-sum $Q_D$ is an obstruction for Yoshikawa moves $\Omega_4$ and $\Omega_4'$.


\section{Marked graph diagrams}
\label{sect-mgd}

In this section, we review the method of describing surface-links by marked graph diagrams.
By a {\it surface-link} $\mathcal L$ we mean mutually disjoint connected and closed (possibly orientable or nonorientable) surfaces smoothly (or piecewise linearly and locally flatly) embedded in the $4$-space $\mathbb R^4$.  When it is connected, we also call it a {\it surface-knot}. Two surface-links are said to be  {\it equivalent} if they are ambient isotopic in $\mathbb R^4$.


A {\it marked graph} is a spatial graph $G$ in $\mathbb R^3$ which satisfies the following:
\begin{itemize}
  \item [(1)] $G$ is a finite regular graph with $4$-valent vertices, say $v_1, v_2, . . . , v_n$.
  \item [(2)] Each $v_i$ is rigid; that is, we fix a rectangular neighborhood $N_i \approx \{(x, y)|-1 \leq x, y \leq 1\},$ where $v_i$ is the origin and the edges incident to $v_i$ are presented by $x^2 = y^2$.
  \item [(3)] Each $v_i$ has a {\it marker}, which is the interval on $N_i$ given by  $\{(x, 0)|-1 \leq x \leq 1\}$.
\end{itemize}

Two  marked graphs are said to be {\it equivalent} if they are ambient isotopic in $\mathbb R^3$ with keeping the rectangular neighborhood and the marker at each vertex.

An {\it orientation} of a marked graph $G$ is a choice of an orientation for each edge of $G$ in such a way that every rigid vertex in $G$ looks like
\xy (-5,5);(5,-5) **@{-},
(5,5);(-5,-5) **@{-},
(3,3.2)*{\llcorner},
(-3,-3.4)*{\urcorner},
(-2.5,2)*{\ulcorner},
(2.5,-2.4)*{\lrcorner},
(3,-0.2);(-3,-0.2) **@{-},
(3,0);(-3,0) **@{-},
(3,0.2);(-3,0.2) **@{-},
\endxy or
\xy (-5,5);(5,-5) **@{-},
(5,5);(-5,-5) **@{-},
(2.5,2)*{\urcorner},
(-2.5,-2.2)*{\llcorner},
(-3.2,3)*{\lrcorner},
(3,-3.4)*{\ulcorner},
(3,-0.2);(-3,-0.2) **@{-},
(3,0);(-3,0) **@{-},
(3,0.2);(-3,0.2) **@{-}
\endxy
, up to rotation.
A marked graph $G$ is said to be {\it orientable} if $G$ admits an orientation. Otherwise, it is said to be {\it non-orientable}. Note that there is a non-orientable marked graph (see Fig.~\ref{fig-nori-mg}). By an {\it oriented marked graph} we mean an orientable marked graph with a fixed orientation.

\begin{figure}[h]
\begin{center}
\resizebox{0.25\textwidth}{!}{%
  \includegraphics{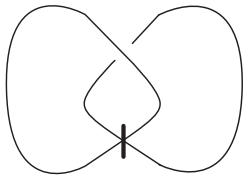}}
\caption{A non-orientable marked graph}\label{fig-nori-mg}
\end{center}
\end{figure}

For $t \in \mathbb R,$ we denote by $\mathbb R^3_t$ the hyperplane of $\mathbb R^4$ whose fourth coordinate
is $t \in \mathbb R$, i.e., $\mathbb R^3_t := \{(x_1, x_2, x_3, x_4) \in
\mathbb R^4~|~ x_4 = t \}$. A surface-link $\mathcal L$ in $\mathbb R^4=\mathbb R^3 \times \mathbb R$ can be described in terms of its {\it cross-sections} $\mathcal L \cap \mathbb R^3_t$ for $t \in \mathbb R$ (cf. \cite{Fox}).
It is  known \cite{Ka1,KSS,KK,Lo} that any surface-link $\mathcal L$ can be deformed into a surface-link $\mathcal L'$, called a {\it hyperbolic splitting} of $\mathcal L$,
by an ambient isotopy of $\mathbb R^4$ in such a way that
the projection $p: \mathcal L' \to \mathbb R$ to the fourth coordinate satisfies the following:
\begin{itemize}
\item[(1)] All critical points are non-degenerate.
\item[(2)] All the index 0 critical points (minimal points) are in $\mathbb R^3_{-1}.$
\item[(3)] All the index 1 critical points (saddle points) are in $\mathbb R^3_0$.
\item[(4)] All the index 2 critical points (maximal points) are in $\mathbb R^3_1$.
\end{itemize}

Let $\mathcal L$ be a surface-link in $\mathbb R^4$ and let ${\mathcal L'}$ be a hyperbolic splitting of $\mathcal L.$
Then the cross-section $\mathcal L' \cap \mathbb R^3_0$ at $t=0$ is a $4$-valent graph in $\mathbb R^3_0$. We give a marker at each $4$-valent vertex (saddle point) that indicates how the saddle point opens up above as shown in Fig.~\ref{sleesan2:fig1}. The resulting marked graph is called a {\it marked graph} presenting $\mathcal L$.
It is usually described by a diagram
on $\mathbb R^2$ called a {\it marked graph diagram} or a {\it ch-diagram} (cf. \cite{So}).

\begin{figure}[h]
\begin{center}
\resizebox{0.60\textwidth}{!}{%
  \includegraphics{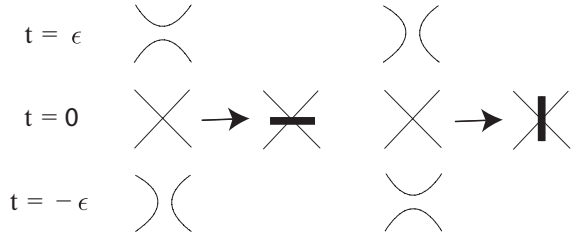} }
\caption{Marking of a vertex} \label{sleesan2:fig1}
\end{center}
\end{figure}

A {\it banded link} $\mathcal {BL}$ in $\mathbb R^3$ is a pair $(L, \mathcal B)$ consisting of a link $L$ in $\mathbb R^3$ and a set $\mathcal B=\{B_1, \ldots, B_n\}$ of mutually disjoint $n$ bands $B_i$ spanning $L$.

Let $\mathcal {BL} =(L, \mathcal B)$ be a banded link.  By an ambient isotopy of $\mathbb R^3$, we shorten the bands so that each band is contained in a small $2$-disk.  Replacing the neighborhood of each band to the neighborhood of a marked vertex
as in Fig.~\ref{fig-orbd-2}, we obtain a marked graph, called a {\it marked graph associated with }$\mathcal {BL}$.
Conversely when a marked graph $G$ in $\mathbb R^3$ is given, by  replacing each marked vertex with a band as in Fig.~\ref{fig-orbd-2}, we obtain a banded link $\mathcal {BL}(G)$, called a {\it banded link associated with } $G$.

\begin{figure}[h]
\begin{center}
\resizebox{0.50\textwidth}{!}{%
  \includegraphics{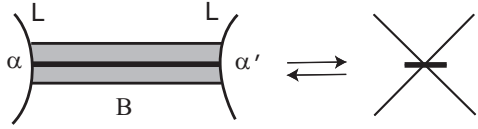} }
\caption{A band and a marked vertex}\label{fig-orbd-2}
\end{center}
\end{figure}

Let $G$ be a marked graph and $\mathcal {BL}(G)$ a banded link associated with $G$.
We denote by $L_-(G)$ the link $L$, and by $L_+(G)$ the link obtained from $L$ by surgery along the bands $\mathcal B$.  Moreover, when $G$ is described by a marked graph diagram $D$, then $\mathcal {BL}(G)$ is also called a {\it banded link associated with} $D$, and denoted by $\mathcal {BL}(D)$. $L_-(G)$ (or $L_+(G)$, resp.) is denoted by $L_-(D)$ (or $L_+(D)$, resp.).  We call $L_-(D)$ the {\it negative resolution} and $L_+(D)$ the {\it positive resolution} of $D$.
A diagram of $L_-(D)$ (or $L_+(D)$, resp.) is obtained from $D$ by smoothing the marked vertices, which we call the
{\it negative resolution diagram} (or {\it positive resolution diagram}, resp.) of $D$.

Fig.~\ref{spun-mgraph-res} shows an example of a marked graph diagram $D$ and its associated banded link,
and the negative and positive resolutions.

A marked graph diagram $D$ is said to be {\it admissible} if both resolutions $L_-(D)$ and $L_+(D)$ are trivial links.

\begin{figure}[h]
\begin{center}
\resizebox{0.75\textwidth}{!}{%
  \includegraphics{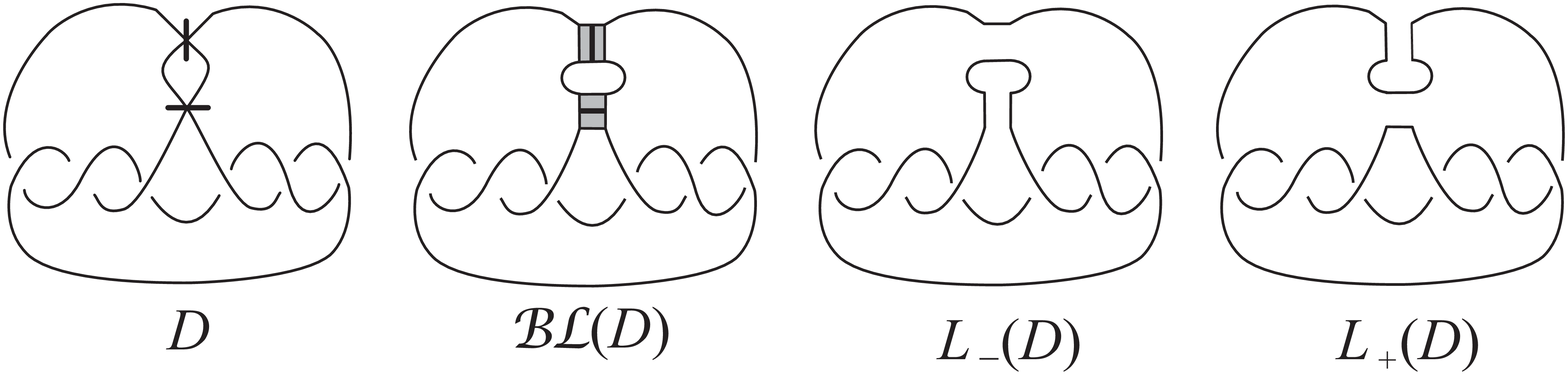} }
\caption{A marked graph diagram, its associated banded link, and resolutions}
\label{spun-mgraph-res}
\end{center}
\end{figure}

We recall how to construct a surface-link from a given admissible marked graph diagram (cf.  \cite{Ka2,Kaw,KSS,Yo}).

Let $D$ be a marked graph diagram.
Let $\mathcal {BL}(D) =(L, \mathcal B)$ be its associated banded link.
Let $\mathcal B$ consist of bands $B_1, \dots, B_n$.
We define a surface $F(D) \subset \mathbb R^3 \times [-1,1]$ by

\begin{equation*}
(\mathbb R^3_t, F(D) \cap \mathbb R^3_t)=\left\{%
\begin{array}{ll}
    (\mathbb R^3, L_+(D)) & \hbox{for $0 < t \leq 1$,} \\
    \bigg( \mathbb R^3, L_-(D) \cup \bigg( \bigcup_{i=1}^n B_i \bigg) \bigg) & \hbox{for $t = 0$,} \\
    (\mathbb R^3, L_-(D)) & \hbox{for $-1 \leq t < 0$.} \\
\end{array}%
\right.
\end{equation*}

We call $F(D)$ the {\it proper surface associated with} $D$.

It is known that $D$ is orientable if and only if $F(D)$ is an orientable surface.  When $D$ is oriented, the resolutions
$L_-(D)$ and $L_+(D)$ have orientations induced from the orientation of $D$, and we assume $F(D)$ is oriented so that the induced orientation on $L_+(D) = \partial F(D) \cap \mathbb R^3_1$ matches the orientation of $L_+(D)$.

When $D$ is admissible,  we can obtain a surface-link from $F(D)$ by attaching trivial disks in $\mathbb R^3 \times [1, \infty)$ and trivial disks in $\mathbb R^3 \times (-\infty, 1]$.  We denote this surface-link by $\mathcal S(D)$, and call it
the {\it surface-link associated with $D$}.

We say that a surface-link $\mathcal L$ is {\it presented} by a marked graph diagram $D$ if  $\mathcal L$ is equivalent to
$\mathcal S(D)$.  Any surface-link can be presented by a marked graph diagram.

{\it Yoshikawa moves} for marked graph diagrams are local moves $\Omega_1, \ldots, \Omega_5$ (Type I) and $\Omega_6, \ldots, \Omega_8$ (Type II) illustrated in Fig.~\ref{fig-moves-type-I} and Fig.~\ref{fig-moves-type-II}.
It is known that two admissible marked graph diagrams present equivalent surface-links if and only if they are related by a finite sequence of Yoshikawa moves \cite{KK,KJL1,Sw}.

\begin{figure}[h]
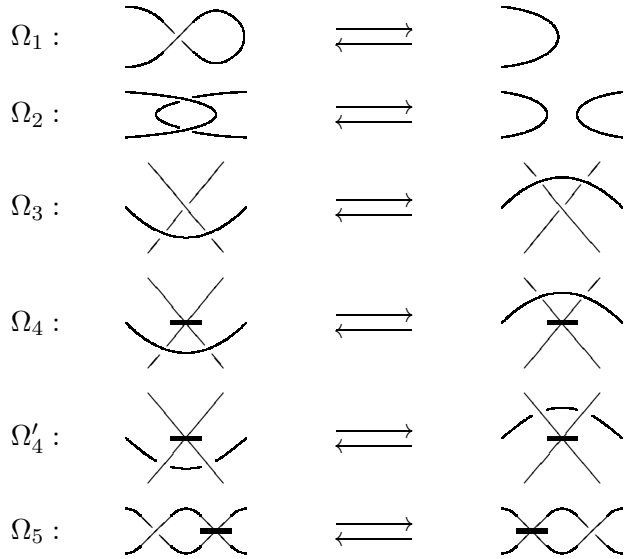

\begin{center}
\centerline{
\xy (12,2);(16,6) **@{-},
(12,6);(13.5,4.5) **@{-},
(14.5,3.5);(16,2) **@{-},
(16,6);(22,6) **\crv{(18,8)&(20,8)},
(16,2);(22,2) **\crv{(18,0)&(20,0)}, (22,6);(22,2) **\crv{(23.5,4)},
(7,8);(12,6) **\crv{(10,8)}, (7,0);(12,2) **\crv{(10,0)},
(35,5);(45,5) **@{-} ?>*\dir{>}, (35,3);(45,3) **@{-} ?<*\dir{<},
(57,8);(57,0) **\crv{(67,7)&(67,1)}, (-5,4)*{\Omega_1 :}, (73,4)*{},
\endxy }

\vskip.3cm

\centerline{ \xy (7,7);(7,1)  **\crv{(23,6)&(23,2)},
(16,6.3);(23,7)**\crv{(19,6.9)}, (16,1.7);(23,1) **\crv{(19,1.1)},
(14,5.7);(14,2.3) **\crv{(8,4)},
(35,5);(45,5) **@{-} ?>*\dir{>}, (35,3);(45,3) **@{-} ?<*\dir{<},
(57,7);(57,1) **\crv{(65,6)&(65,2)}, (73,7);(73,1)
**\crv{(65,6)&(65,2)},
(-5,4)*{\Omega_2 :},
\endxy}

\vskip.3cm

\centerline{
\xy (7,6);(23,6) **\crv{(15,-2)},
(10,0);(11.5,1.8) **@{-},
(17.5,3);(14.5,6.6) **@{-},
(14.5,6.6);(10,12) **@{-},
(20,12);(15.5,6.6) **@{-},
(14.5,5.5);(12.5,3) **@{-},
(18.5,1.8);(20,0) **@{-},
(35,7);(45,7) **@{-} ?>*\dir{>},
(35,5);(45,5) **@{-} ?<*\dir{<},
(57,6);(73,6) **\crv{(65,14)},
(70,12);(68.5,10.2) **@{-},
(67.5,9);(65.5,6.5) **@{-},
(64.6,5.5);(60,0) **@{-},
(62.5,9);(64.4,6.6) **@{-},
(64.4,6.6);(70,0) **@{-},
(61.5,10.2);(60,12) **@{-},
(-5,6)*{\Omega_3:},
\endxy}

\vskip.3cm

 \centerline{ \xy
 (7,6);(23,6)  **\crv{(15,-2)},
 (10,0);(11.5,1.8) **@{-},
(12.5,3);(20,12) **@{-},
(10,12);(17.5,3) **@{-},
(18.5,1.8);(20,0) **@{-},
(13,6);(17,6) **@{-}, (13,6.1);(17,6.1) **@{-}, (13,5.9);(17,5.9)
**@{-}, (13,6.2);(17,6.2) **@{-}, (13,5.8);(17,5.8) **@{-},
(35,7);(45,7) **@{-} ?>*\dir{>},
(35,5);(45,5) **@{-} ?<*\dir{<},
(57,6);(73,6)  **\crv{(65,14)},
(70,12);(68.5,10.2) **@{-},
(67.5,9);(60,0) **@{-},
(70,0);(62.5,9) **@{-},
(61.5,10.2);(60,12) **@{-},
(63,6);(67,6) **@{-}, (63,6.1);(67,6.1) **@{-}, (63,5.9);(67,5.9)
**@{-}, (63,6.2);(67,6.2) **@{-}, (63,5.8);(67,5.8) **@{-},
(-5,6)*{\Omega_4:},
\endxy}

\vskip.3cm

 \centerline{ \xy
  (13,2.2);(17,2.2)  **\crv{(15,1.7)},
  (7,6);(11,3)  **\crv{(10,3.5)},
  (23,6);(19,3)  **\crv{(20,3.5)},
 (10,0);(20,12) **@{-},
(10,12);(20,0) **@{-},
(13,6);(17,6) **@{-}, (13,6.1);(17,6.1) **@{-}, (13,5.9);(17,5.9)**@{-}, (13,6.2);(17,6.2) **@{-}, (13,5.8);(17,5.8) **@{-},
(35,7);(45,7) **@{-} ?>*\dir{>},
(35,5);(45,5) **@{-} ?<*\dir{<},
   (63,9.8);(67,9.8)  **\crv{(65,10.3)},
  (57,6);(61,9)  **\crv{(60,8.5)},
  (73,6);(69,9)  **\crv{(70,8.5)},
(70,12);(60,0) **@{-},
(70,0);(60,12) **@{-},
(63,6);(67,6) **@{-}, (63,6.1);(67,6.1) **@{-}, (63,5.9);(67,5.9)
**@{-},(63,6.2);(67,6.2) **@{-}, (63,5.8);(67,5.8) **@{-},
(-5,6)*{\Omega_4':},
\endxy}

\vskip.3cm

 \centerline{ \xy (9,2);(13,6) **@{-}, (9,6);(10.5,4.5) **@{-},
(11.5,3.5);(13,2) **@{-}, (17,2);(21,6) **@{-}, (17,6);(21,2)
**@{-}, (13,6);(17,6) **\crv{(15,8)}, (13,2);(17,2) **\crv{(15,0)},
(7,7);(9,6) **\crv{(8,7)}, (7,1);(9,2) **\crv{(8,1)}, (23,7);(21,6)
**\crv{(22,7)}, (23,1);(21,2) **\crv{(22,1)},
(17,4);(21,4) **@{-}, (17,4.1);(21,4.1) **@{-}, (17,3.9);(21,3.9)
**@{-}, (17,4.2);(21,4.2) **@{-}, (17,3.8);(21,3.8) **@{-},
(35,5);(45,5) **@{-} ?>*\dir{>}, (35,3);(45,3) **@{-} ?<*\dir{<},
(59,2);(63,6) **@{-}, (59,6);(63,2) **@{-}, (67,2);(71,6) **@{-},
(67,6);(68.5,4.5) **@{-}, (69.5,3.5);(71,2) **@{-}, (63,6);(67,6)
**\crv{(65,8)}, (63,2);(67,2) **\crv{(65,0)}, (57,7);(59,6)
**\crv{(58,7)}, (57,1);(59,2) **\crv{(58,1)}, (73,7);(71,6)
**\crv{(72,7)}, (73,1);(71,2) **\crv{(72,1)},
(63,4);(59,4) **@{-}, (63,4.1);(59,4.1) **@{-}, (63,3.9);(59,3.9)
**@{-}, (63,4.2);(59,4.2) **@{-}, (63,3.8);(59,3.8) **@{-},
 (-5,4)*{\Omega_5:},
\endxy }
\caption{ Moves of Type I}\label{fig-moves-type-I}
\end{center}
\end{figure}

\begin{figure}[h]
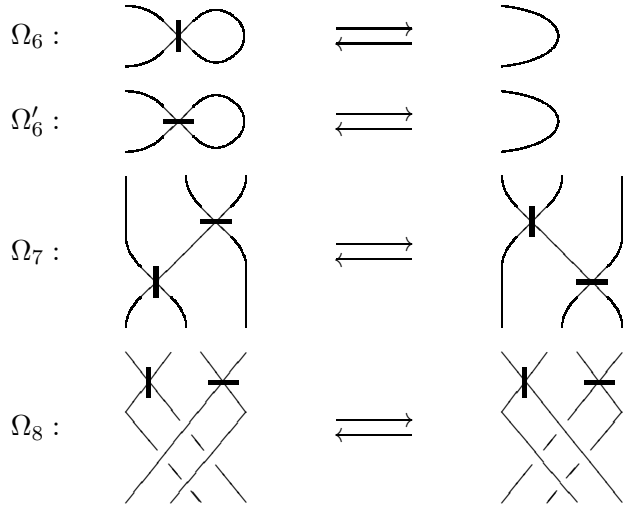

\begin{center}
\centerline{ \xy (12,6);(16,2) **@{-}, (12,2);(16,6) **@{-},
(16,6);(22,6) **\crv{(18,8)&(20,8)}, (16,2);(22,2)
**\crv{(18,0)&(20,0)}, (22,6);(22,2) **\crv{(23.5,4)}, (7,8);(12,6)
**\crv{(10,8)}, (7,0);(12,2) **\crv{(10,0)},
(35,5);(45,5) **@{-} ?>*\dir{>}, (35,3);(45,3) **@{-} ?<*\dir{<},
(57,8);(57,0) **\crv{(67,7)&(67,1)}, (-5,4)*{\Omega_6 :}, (73,4)*{},
(14,6);(14,2) **@{-}, (14.1,6);(14.1,2) **@{-}, (13.9,6);(13.9,2)
**@{-}, (14.2,6);(14.2,2) **@{-}, (13.8,6);(13.8,2) **@{-},
\endxy}

\vskip.3cm

\centerline{ \xy (12,6);(16,2) **@{-}, (12,2);(16,6) **@{-},
(16,6);(22,6) **\crv{(18,8)&(20,8)}, (16,2);(22,2)
**\crv{(18,0)&(20,0)}, (22,6);(22,2) **\crv{(23.5,4)}, (7,8);(12,6)
**\crv{(10,8)}, (7,0);(12,2) **\crv{(10,0)},
(35,5);(45,5) **@{-} ?>*\dir{>}, (35,3);(45,3) **@{-} ?<*\dir{<},
(57,8);(57,0) **\crv{(67,7)&(67,1)}, (-5,4)*{\Omega'_6 :},
(73,4)*{}, (12,4);(16,4) **@{-}, (12,4.1);(16,4.1) **@{-},
(12,4.2);(16,4.2) **@{-}, (12,3.9);(16,3.9) **@{-},
(12,3.8);(16,3.8) **@{-},
\endxy}

\vskip.3cm

\centerline{ \xy (9,4);(17,12) **@{-}, (9,8);(13,4) **@{-},
(17,12);(21,16) **@{-}, (17,16);(21,12) **@{-}, (7,0);(9,4)
**\crv{(7,2)}, (7,12);(9,8) **\crv{(7,10)}, (15,0);(13,4)
**\crv{(15,2)}, (17,16);(15,20) **\crv{(15,18)}, (21,16);(23,20)
**\crv{(23,18)}, (21,12);(23,8) **\crv{(23,10)}, (7,12);(7,20)
**@{-}, (23,8);(23,0) **@{-},
(11,4);(11,8) **@{-},
(10.9,4);(10.9,8) **@{-},
(11.1,4);(11.1,8) **@{-},
(10.8,4);(10.8,8) **@{-},
(11.2,4);(11.2,8) **@{-},
(17,14);(21,14) **@{-},
(17,14.1);(21,14.1) **@{-},
(17,13.9);(21,13.9) **@{-},
(17,14.2);(21,14.2) **@{-},
(17,13.8);(21,13.8) **@{-},
(35,11);(45,11) **@{-} ?>*\dir{>}, (35,9);(45,9) **@{-} ?<*\dir{<},
(71,4);(63,12) **@{-}, (71,8);(67,4) **@{-}, (63,12);(59,16) **@{-},
(63,16);(59,12) **@{-}, (73,0);(71,4) **\crv{(73,2)}, (73,12);(71,8)
**\crv{(73,10)}, (65,0);(67,4) **\crv{(65,2)}, (63,16);(65,20)
**\crv{(65,18)}, (59,16);(57,20) **\crv{(57,18)}, (59,12);(57,8)
**\crv{(57,10)}, (73,12);(73,20) **@{-}, (57,8);(57,0) **@{-},
(61,12);(61,16) **@{-},
(61.1,12);(61.1,16) **@{-},
(60.9,12);(60.9,16) **@{-},
(61.2,12);(61.2,16) **@{-},
(60.8,12);(60.8,16) **@{-},
(67,6);(71,6) **@{-},
(67,6.1);(71,6.1) **@{-},
(67,5.9);(71,5.9) **@{-},
(67,6.2);(71,6.2) **@{-},
(67,5.8);(71,5.8) **@{-},
(-5,10)*{\Omega_7:},
 \endxy}

\vskip.3cm

\centerline{ \xy (7,20);(14.2,11) **@{-}, (15.8,9);(17.4,7) **@{-},
(19,5);(23,0) **@{-}, (13,20);(7,12) **@{-}, (7,12);(11.2,7) **@{-},
(12.7,5.2);(14.4,3.2) **@{-}, (15.7,1.6);(17,0) **@{-},
(17,20);(23,12) **@{-}, (13,0);(23,12) **@{-}, (7,0);(23,20) **@{-},
(10,18);(10,14) **@{-}, (10.1,18);(10.1,14) **@{-},
(9.9,18);(9.9,14) **@{-}, (10.2,18);(10.2,14) **@{-},
(9.8,18);(9.8,14) **@{-}, (18,16);(22,16) **@{-},
(18,16.1);(22,16.1) **@{-}, (18,15.9);(22,15.9) **@{-},
(18,16.2);(22,16.2) **@{-}, (18,15.8);(22,15.8) **@{-},
(35,11);(45,11) **@{-} ?>*\dir{>}, (35,9);(45,9) **@{-} ?<*\dir{<},
(73,20);(65.8,11) **@{-}, (64.2,9);(62.6,7) **@{-}, (61,5);(57,0)
**@{-}, (67,20);(73,12) **@{-}, (73,12);(68.8,7) **@{-},
(67.3,5.2);(65.6,3.2) **@{-}, (64.3,1.6);(63,0) **@{-},
(63,20);(57,12) **@{-}, (67,0);(57,12) **@{-}, (73,0);(57,20)
**@{-},
(60,18);(60,14) **@{-}, (60.1,18);(60.1,14) **@{-},
(59.9,18);(59.9,14) **@{-}, (60.2,18);(60.2,14) **@{-},
(59.8,18);(59.8,14) **@{-}, (68,16);(72,16) **@{-},
(68,16.1);(72,16.1) **@{-}, (68,15.9);(72,15.9) **@{-},
(68,16.2);(72,16.2) **@{-}, (68,15.8);(72,15.8) **@{-},
(-5,10)*{\Omega_8:},
\endxy}
\caption{Moves of Type II}\label{fig-moves-type-II}
\end{center}
\end{figure}


\section{Lipson's state model for classical links}
\label{sect-lip-link}
We first explain Lipson's state models for classical link invariants from \cite{ASL}. Some notations in this section are different from \cite{ASL}.
Let $K$ be an unoriented link and $D$ a diagram of $K$.
Let $\widehat{D}$ be the diagram obtained from $D$ by removing all crossings of $D$ as illustrated in Fig.~\ref{fig-cr-rem}.

\begin{figure}[h]
\begin{center}
\resizebox{0.4\textwidth}{!}{%
  \includegraphics{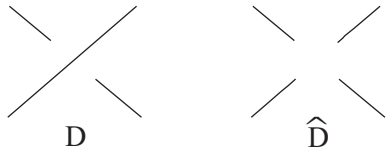} }
\caption{A diagram removed crossings}\label{fig-cr-rem}
\end{center}
\end{figure}

A {\it state} or a {\it labelling}
is a map $\sigma:\{$connected components of $\widehat{D} \} \rightarrow \{0,1\}$.
 A state (or a labelling) is {\it legal}
if at each crossing, each label occurs an even number of times.  See  Fig.~\ref{fig-legal}, where $a$ and $b$ are to be interpreted as distinct labels.
For a legal state $\sigma$ and a crossing $c$ of $D$, we define $\nu(c) = \nu(c,\sigma) \in \{x, y, z, w\}$
as shown in Fig.~\ref{fig-legal}.

\begin{figure}[h]
\begin{center}
\resizebox{0.6\textwidth}{!}{%
  \includegraphics{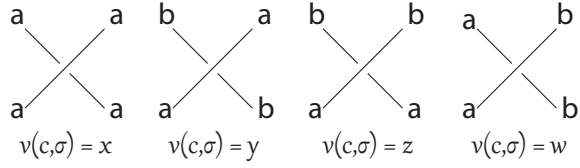} }
\caption{Legal states}\label{fig-legal}
\end{center}
\end{figure}

Define $[D]$ by $$[D] = [D](x,y,z,w) = \sum_{\text{legal states}~\sigma} \prod_{crossings ~c} \nu(c,\sigma)
\in \mathbb Z[x,y,z,w],$$
where $\sigma$ runs over all legal states of $D$ and $c$ runs over all crossings of $D$.

Lipson \cite{ASL} observed that to obtain invariance of $[D]$ under Reidemeister move $\Omega_2$ we are led to the relations:
$$x^2+zw=1, \quad x(z+w)=0, \quad y^2+zw=1, \quad y(z+w)=0,$$
and these relations are also sufficient to obtain invariance under Reidemeister move $\Omega_3$. As a consequence, we have the following.

\begin{prop}[Lipson \cite{ASL}]\label{prop1}
There are two cases in each of which the value $[D]$ is invariant under Reidemeister moves $\Omega_2$ and $\Omega_3$:
\begin{itemize}
  \item [(R1)] $x=y=0$ and $w=z^{-1}$,
  \item [(R2)] $z=-w$ and $x^2=y^2=1+z^2$.
\end{itemize}
\end{prop}

For invariance under Reidemeister move $\Omega_1$, we need to normalize $[D]$.

First we consider a case where $K$ is oriented.
We denote by $w(D)$ the writhe of $D$, that is the number of positive crossings of $D$ minus that of negative ones.

\begin{thm}[Lipson \cite{ASL}]\label{lipson-thm}
When we define $R^{\rm ori}_D(z)$ and $S^{\rm ori}_D(z)$ by below, they are invariants of oriented links. Moreover, $R^{\rm ori}_D$ is equal to $S^{\rm ori}_D$.
\begin{itemize}
  \item [(L1)] $R^{\rm ori}_D(z)=z^{w(D)} [D](0,0,z,z^{-1}) \in \mathbb Z[z, z^{-1}]$,
  \item [(L2)] $S^{\rm ori}_D(z)=z^{w(D)} [D](\frac{z+z^{-1}}{2},\frac{z+z^{-1}}{2},\frac{z-z^{-1}}{2},-\frac{z-z^{-1}}{2})
   \in \frac{1}{2}\mathbb Z[z, z^{-1}]$.
\end{itemize}
\end{thm}

\begin{rem}
Our $R^{\rm ori}_D(z)$ and $S^{\rm ori}_D(z)$  are denoted by $R_D(z)$ and $S_D(z)$ in \cite{ASL}.
Let $F_D(a,z)$ be the Kauffman polynomial and let $N$ be the number of the components of the link presented by $D$. Lipson \cite{ASL} proved that $R^{\rm ori}_D(z)=(-1)^{N-1}F({iz^{-1},iz-iz^{-1}}).$ Thus $R^{\rm ori}_D$ and $S^{\rm ori}_D$ are two distinct state models for the same link invariant derived from $F_D(a,z)$.
\end{rem}

Next we consider a case where $K$ is unoriented.
We denote by $sw(D)$ the self-writhe of $D$.  (Let $K= K_1 \cup \cdots \cup K_N$, where $K_i$ $(i=1, \dots, N)$ is the component of $K$, and let $D= D_1 \cup \cdots \cup D_N$, where $D_i$ $(i=1, \dots, N)$ is the sub-diagram corresponding to $K_i$.  Note that the writhe $w(D_i)$ is well defined for any orientation of $K_i$.  The self-writhe is the sum of $w(D_i)$ for all $i$.)

\begin{thm}\label{lipson2-thm}
When we define $R^{\rm unori}_D(z)$ and $S^{\rm unori}_D(z)$ by below, they are invariants of unoriented links. Moreover, $R^{\rm unori}_D$ is equal to $S^{\rm unori}_D$.
\begin{itemize}
  \item [(L1)] $R^{\rm unori}_D(z)=z^{sw(D)} [D](0,0,z,z^{-1})$,
  \item [(L2)] $S^{\rm unori}_D(z)=z^{sw(D)} [D](\frac{z+z^{-1}}{2},\frac{z+z^{-1}}{2},\frac{z-z^{-1}}{2},-\frac{z-z^{-1}}{2})$.
\end{itemize}
\end{thm}

\begin{proof}
Considering the difference of $[D](0,0,z,z^{-1})$ by a Reidemeister move $\Omega_1$, we see that $z^{sw(D)}$ works as normalization.  Thus $R^{\rm unori}_D(z)$ is an invariant of unoriented links.
Give an orientation to $K$, then $w(D) = sw(D) + 2 {\rm lk}(D)$, where
${\rm lk}(D)$ is the total linking number $\sum_{i < j} {\rm lk}(D_i, D_j)$.
Then $R^{\rm ori}_D(z)=z^{2 {\rm lk}(D)} R^{\rm unori}_D(z)$ and
$S^{\rm ori}_D(z)=z^{2 {\rm lk}(D)} S^{\rm unori}_D(z)$.  Thus $R^{\rm unori}_D(z) = S^{\rm unori}_D(z)$.
\end{proof}

\section{Lipson's state model for marked graph diagrams}
\label{sect-mgd-ym}

In this section, we generalize Lipson's state models of classical link diagrams to marked graph diagrams, and observe when we obtain state-sum invariants of surface-links. Throughout this section, marked graph diagrams are unoriented.

Let $D$ be a marked graph diagram.
Let $\widehat{D}$ be the diagram obtained from $D$ by removing all crossings of $D$ as illustrated in Fig.~\ref{fig-cr-rem}. (We do not remove marked vertices of $D$.)
A {\it state} or a {\it labelling}
is a map $\sigma:\{$connected components of $\widehat{D} \} \rightarrow \{0,1\}$.
 A state (or a labelling) is {\it legal}
if at each crossing, each label occurs an even number of times. For a legal state $\sigma$ and a classical crossing $c$, we define $\nu(c) = \nu(c, \sigma) \in \{x, y, z, w\}$ as shown in Fig.~\ref{fig-legal}.

Define $[D]$ by $$[D] = [D](x,y,z,w) = \sum_{\text{legal states}~\sigma} \prod_{\text{crossings}~c} \nu(c,\sigma) \in \mathbb Z[x,y,z,w].$$

\bigskip
Now one observes invariance of $[D]$ under Yoshikawa moves.

\begin{lem}\label{lem1}
The state-sum $[D]$ is invariant under Yoshikawa moves $\Omega_5,$ $\Omega_6,$ $\Omega_6'$ and $\Omega_7$.
\end{lem}

\begin{proof}
Let $D$ and $D'$ be related by a Yoshikawa move $\Omega_5$ and let $B$ be a $2$-disk in $\mathbb R^2$ where the Yoshikawa move is applied.
There are two possible cases of legal labellings for $D \cap B$ as shown in Fig.~\ref{smy5-1}.
It is easily verified that Yoshikawa move $\Omega_5$ has no effect upon $[D]$ (see Fig.~\ref{smy5-1}).

\begin{figure}[h]
\begin{center}
\resizebox{0.6\textwidth}{!}{%
  \includegraphics{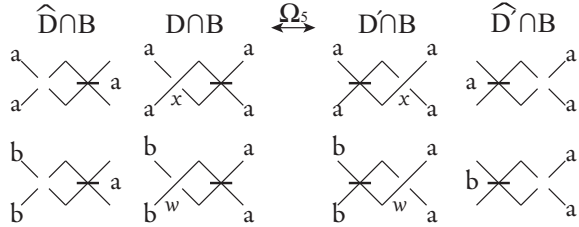} }
\caption{Yoshikawa move $\Omega_5$}\label{smy5-1}
\end{center}
\end{figure}

Let $B$ be a 2-disk where a Yoshikawa move $\Omega_6, \Omega_6'$ or $\Omega_7$ is applied on $D$, and let $D'$ be the result. Since there are no crossings in $D \cap B$ and $D' \cap B$, and since $D \cap B$ and $D' \cap B$ are connected, it is obvious that $[D]=[D']$.
\end{proof}

\begin{lem}\label{lem2}
The following $p_1, \ldots, p_4$ are relations  for the polynomial $[D]$ to be  invariant under Yoshikawa moves $\Omega_4$ and $\Omega_4'$:
$$
p_1 \/ :   xz=0,    \quad    p_2  \/ :   xw=0,  \quad
p_3 \/ :   yz=0,    \quad    p_4 \/ :   yw=0.
$$
\end{lem}

\begin{proof}
We need to check that the polynomial $[D]$ is equal to the polynomial $[D']$, where $D'$ is the diagram obtained from $D$ by applying a single Yoshikawa move $\Omega_4$ in a $2$-disk $B$ as shown in Fig.~\ref{smy4}:\\
\begin{figure}[h]
\begin{center}
\resizebox{0.6\textwidth}{!}{%
  \includegraphics{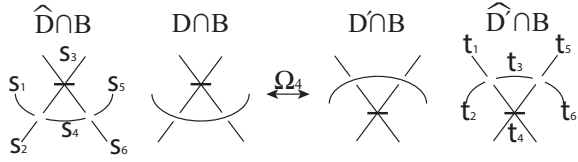} }
\caption{Yoshikawa move $\Omega_4$}\label{smy4}
\end{center}
\end{figure}

There are four possible cases of labellings that are legal for both  $D \cap B$ and $D' \cap B$ as shown
 in Fig.~\ref{smy4-1}. For these cases, it is obvious that $[D]=[D']$.
 In Fig.~\ref{smy4-4}, we show possible cases of legal labellings for $D\cap B$ (or $D'\cap B$) that have no counterparts for $D'\cap B$ (or $D\cap B$, resp.). In these cases, we obtain the relations $p_1, \dots, p_4$ for $[D]=[D']$.
 Thus we obtain the four relations $p_1, \ldots, p_4$ for the polynomial $[D]$ to be invariant under Yoshikawa move $\Omega_4$.
Similarly, we obtain the same four relations $p_1, \ldots, p_4$ for the polynomial $[D]$ to be invariant under Yoshikawa move $\Omega_4'$.
 \begin{figure}[h]
\begin{center}
\resizebox{0.6\textwidth}{!}{%
  \includegraphics{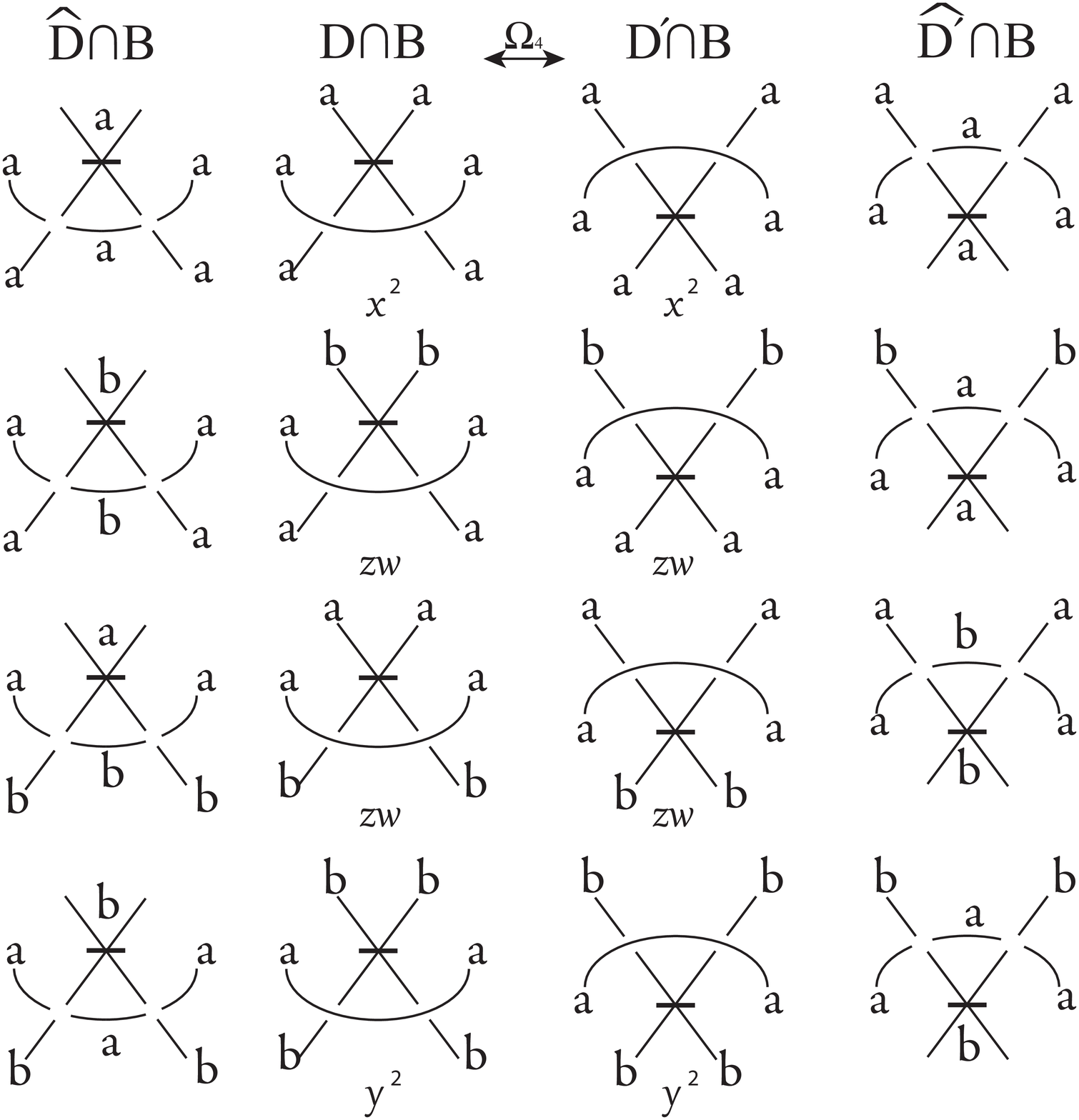} }
\caption{Yoshikawa move $\Omega_4$ (1)}\label{smy4-1}
\end{center}
\end{figure}

\begin{figure}[h]
\begin{center}
\resizebox{0.8\textwidth}{!}{%
  \includegraphics{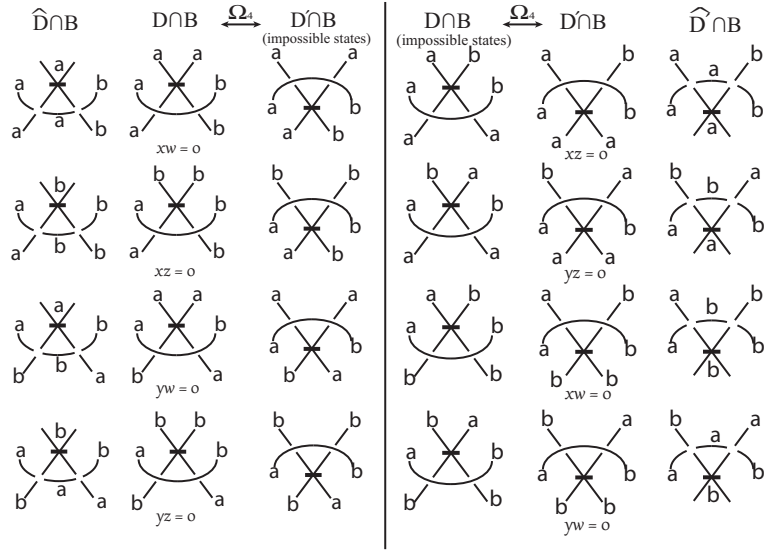} }
\caption{Yoshikawa move $\Omega_4$ (2)}\label{smy4-4}
\end{center}
\end{figure}

\end{proof}

\begin{lem}\label{lem2}
The following $q_1, \ldots, q_7$ are relations  for the polynomial $[D]$ to be invariant under Yoshikawa move $\Omega_8$:
\begin{align*}
      &q_1 : (x^2+y^2)z^2=(x^2+y^2)w^2, \\
      &q_2 : (x^2+w^2)yz=(x^2+z^2)yw, \\
      &q_3 : (x^2+w^2)xz=(x^2+z^2)xw,  \\
      &q_4 : (y^2+z^2)xz=(y^2+w^2)xw, \\
      &q_5 : (x^2+w^2)yw=(x^2+z^2)yz, \\
      &q_6 : (y^2+z^2)xw=(y^2+w^2)xz,  \\
      &q_7 : (y^2+z^2)yw=(y^2+w^2)yz.
\end{align*}
\end{lem}

\begin{proof}
Possible legal labellings are as shown in Fig.~\ref{smy8-1} and \ref{smy8-5}.
\begin{figure}[ht]
\begin{center}
\resizebox{0.6\textwidth}{!}{%
  \includegraphics{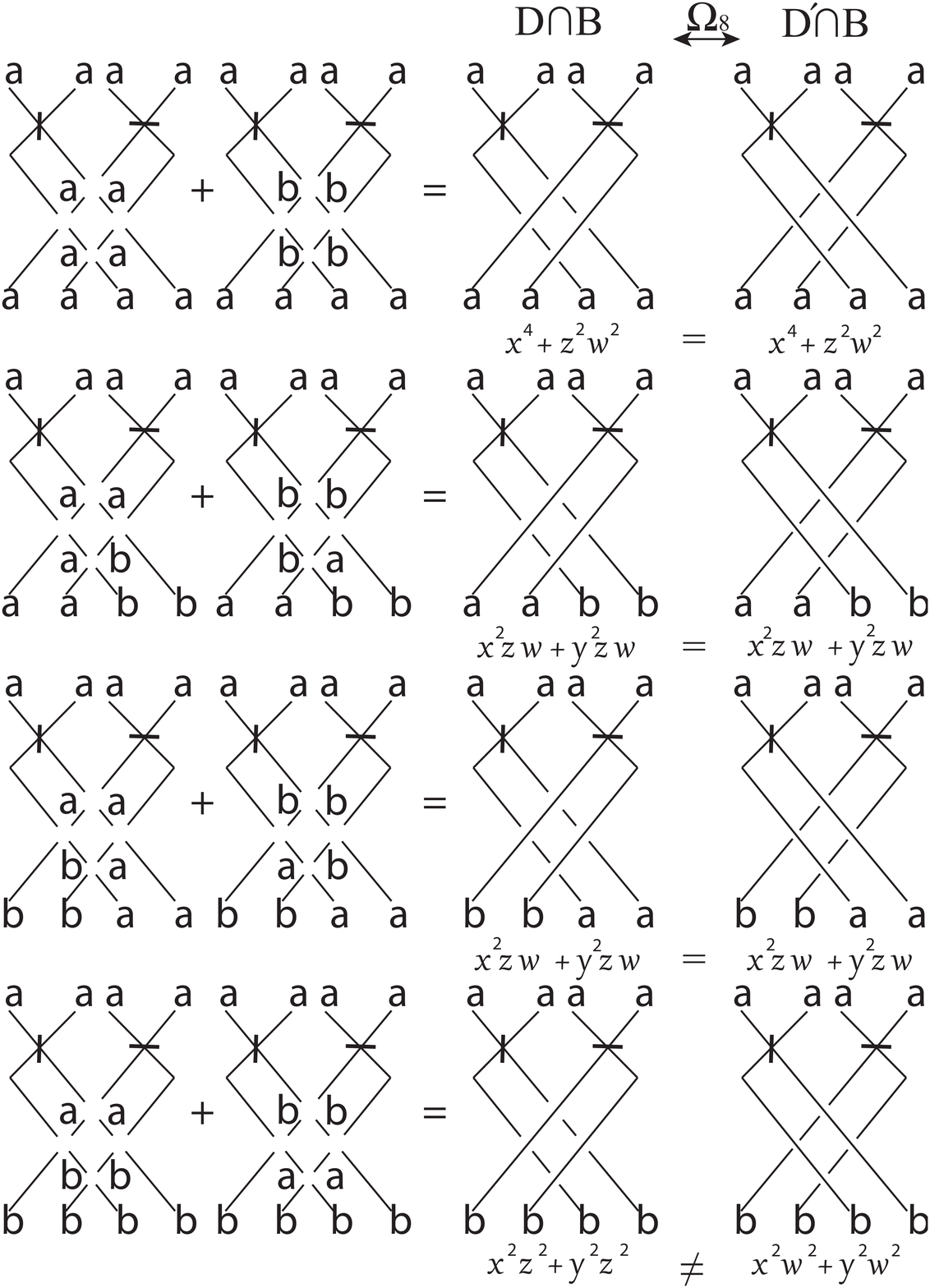} }
\caption{Yoshikawa move $\Omega_8$ (1)}\label{smy8-1}
\end{center}
\end{figure}

\begin{figure}[h]
\begin{center}
\resizebox{0.8\textwidth}{!}{%
  \includegraphics{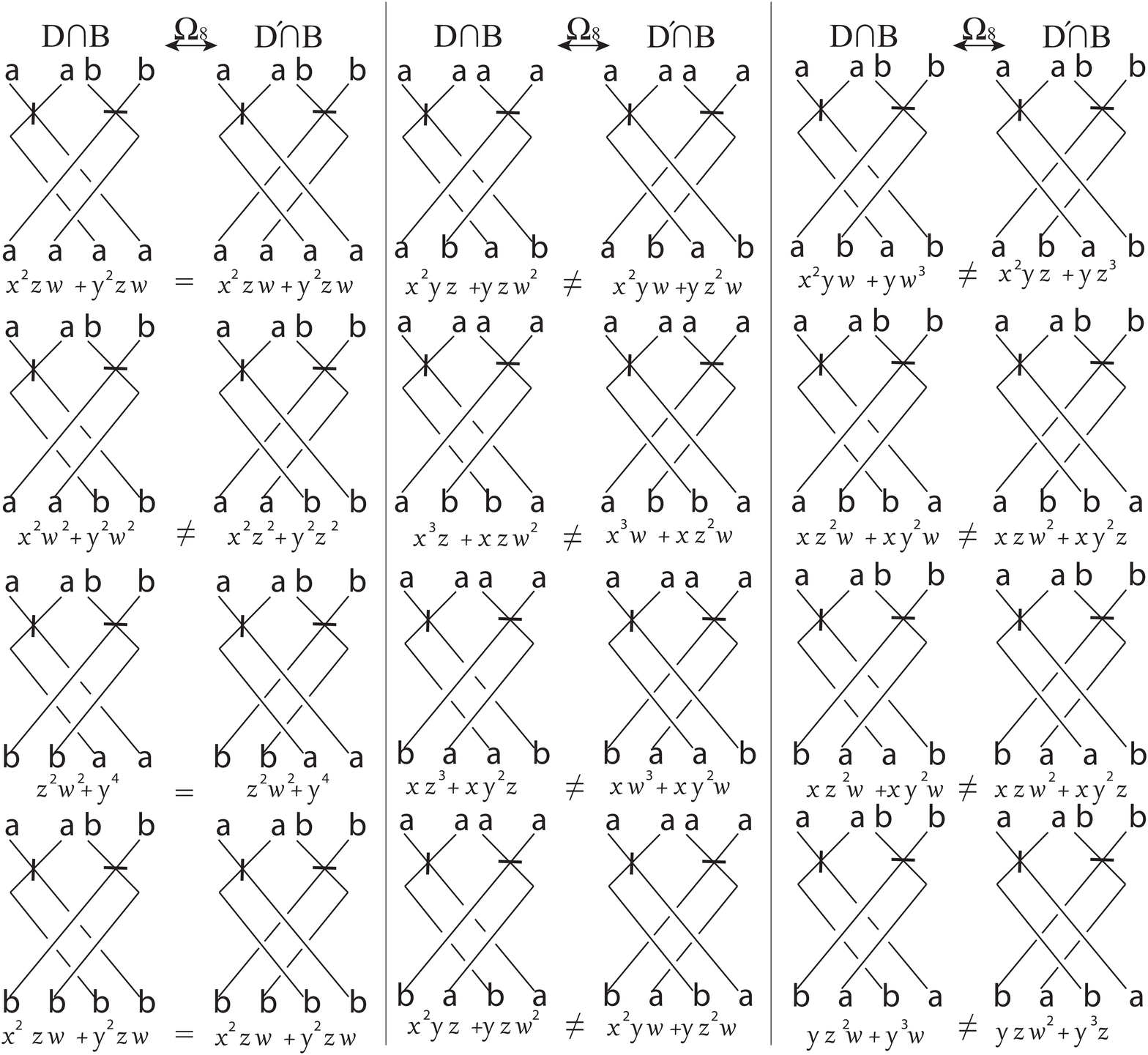} }
\caption{Yoshikawa move $\Omega_8$ (2)}\label{smy8-5}
\end{center}
\end{figure}

From these figures, to obtain invariance of the polynomial $[D]$ under Yoshikawa move $\Omega_8$ we are led to the seven relations $q_1, \ldots, q_7$.
\end{proof}

Let $D$ be a marked graph diagram, and let $D_+ = L_{+}(D)$ be the positive resolution diagram of $D$, i.e., the link diagram obtained from $D$ by smoothing at each marked vertex along the marker (see Fig.~\ref{spun-mgraph-res}).
We denote by $t_{+}(D)$ the self-writhe $sw(D_+)$ of $D_+$.

\begin{thm}\label{thm1}
For a marked graph diagram $D$,
define $R_{D}'(z)$ by
$$ R_{D}'(z) = z^{t_{+}(D)}[D](0,0,z,z^{-1}) \in \mathbb Z[z, z^{-1}].$$
Then $R_{D}'(z)$ is invariant under all Yoshikawa moves.
\end{thm}

For a surface-link $\mathcal L$, take a marked graph diagram presenting $\mathcal L$, say $D$.  Then $R_{D}'(z) $ is an invariant of $\mathcal L$.

We note that if $D$ is a marked graph diagram without marked vertices, i.e., $D$ is a link diagram, then $D_+=D$ and
$R_{D}'(z)=R^{\rm unori}_D(z)$.   Thus $R_{D}'$ is a generalization to a surface-link invariant of Lipson's link invariant $R^{\rm unori}_D(z)$.

\begin{proof}
By the same argument with the proof of Theorem~\ref{lipson-thm}, we see  that the polynomial $R'_D(z)$ is invariant under Yoshikawa moves  $\Omega_1, \Omega_2$ and $\Omega_3$. By Lemma~\ref{lem1}, the polynomial $[D](x,y,z,w)$ is invariant under Yoshikawa moves $\Omega_5$, $\Omega_6$, $\Omega_6'$ and $\Omega_7$.
Since $x=y=0,$  the relations $p_1, \ldots, p_4$ and $q_1, \ldots, q_7$ hold, and hence $[D](x,y,z,w)$ is invariant under Yoshikawa moves $\Omega_4, \Omega_4'$ and $\Omega_8$.  Since $t_{+}(D)$ is preserved under these Yoshikawa moves, we see the result.
\end{proof}

\begin{thm} \label{thm2}
Let $D$ be a marked graph diagram.
Define $$S_{D}'=[D](1,1,0,0).$$
Then $S_{D}'$ is invariant under all Yoshikawa moves.
\end{thm}

For a surface-link $\mathcal L$, take a marked graph diagram presenting $\mathcal L$, say $D$.  Then $S_{D}'$ is an invariant of $\mathcal L$.

We note that if $D$ is a marked graph diagram without marked vertices, i.e., $D$ is a link diagram, then $S_{D}'=S^{\rm unori}_{D}(1)$. Thus $S_{D}'$ is a generalization to a surface-link invariant of Lipson's  link invariant $S^{\rm unori}_D(z)$ evaluated with $z=1$.

\begin{proof}
By Theorem~\ref{lipson2-thm}, we see that $S_{D}'=S^{\rm unori}_{D}(1)$ is invariant under Yoshikawa moves $\Omega_1, \Omega_2$ and $\Omega_3$.
By Lemma~\ref{lem1},  the polynomial $[D](x,y,z,w)$ is  invariant under Yoshikawa moves $\Omega_5$, $\Omega_6$, $\Omega_6'$ and $\Omega_7$.
Since $z=0, w=0,$  the relations $p_1, \ldots, p_4$ and $q_1, \ldots, q_7$ hold, and hence $[D](x,y,z,w)$ is
invariant under Yoshikawa moves $\Omega_4, \Omega_4'$ and $\Omega_8$.  Since $t_{+}(D)$ is preserved under these Yoshikawa moves, we see the result.
\end{proof}


\section{On invariants $R_{D}'$ and $S_{D}'$}


We first prove the following.

\begin{thm}\label{thm:R}
Let $\mathcal L$ be a surface-link, and $D$ a marked graph diagram presenting $\mathcal L$.
\begin{itemize}
\item[(1)] If $\mathcal L$ is orientable, then
$R_{D}'(z)=2^N$, where $N$ is the number of connected components of $\mathcal L$.
\item[(2)] If $\mathcal L$ is non-orientable, then
$R_{D}'(z)=0$.
\end{itemize}
\end{thm}

This theorem shows that the invariant $R_{D}'(z)$ can detect orientability of a surface-link.

\bigskip

Let $D$ be a marked graph diagram.  By $|D|$ we denote the $4$-valent graph in $\mathbb R^2$ obtained from $D$ by removing markers and by assuming crossings to be vertices of valency $4$.   Let $v$ be a point of $|D|$.

By a {\it companion loop} $\ell$ on the diagram $D$ {\it with base point }$v$,
we mean a path  $\ell: [0,1] \to |D|$ with $\ell(0)= \ell(1)= v$
such that
(i) when $\ell$ meets a crossing of $D$, it must go straight and that (ii) when $\ell$ meets a marked vertex of $D$, it must turn right or left (Fig.~\ref{ori-pass}), and (iii) $\ell$ does not go through any edge of $|D|$ twice.
(When the base point $v$ is a marked vertex of $D$, we do not assume the condition (ii) there.  A companion loop may passes a crossing or a  marked vertex twice.)

\begin{figure}[h]
\begin{center}
\resizebox{0.6\textwidth}{!}{%
  \includegraphics{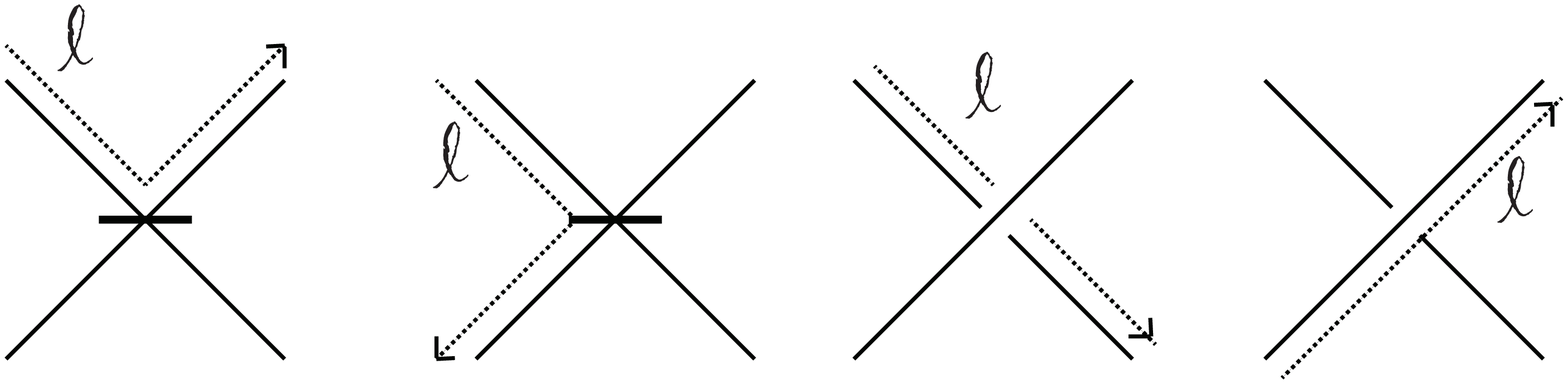} }
\caption{Rules of a companion loop}\label{ori-pass}
\end{center}
\end{figure}

When $D$ is an oriented marked graph diagram,  any companion loop $\ell$ satisfies either that (i)
for any arc on $D$ on which $\ell$ runs,
the orientation of the arc of $D$ is in the same direction of $\ell$ or that (ii)
for any arc on $D$ on which $\ell$ runs,
the orientation of the arc of $D$ is in the opposite direction of $\ell$.
In the former case, we call $\ell$ to be {\it compatible} with respect to the orientation of $D$.

\begin{lem} \label{lem-1}
Let $D$ be an oriented marked graph diagram. Let $\ell$ be a companion loop on $D$. The number of classical crossings where $\ell$ passes is even. (When $\ell$ passes a crossing twice, we count it twice.)
\end{lem}

\begin{proof}
The orientation of $D$ induces orientations of the edges of the $4$-valent graph $|D|$.  Let $E$ be the oriented graph in $\mathbb R^2$ obtained from $|D|$ by removing the edges where $\ell$ passes.  Then $E$ is regarded as a $1$-cycle of $\mathbb R^2$ in the sense of singular homology theory.
Modifying $\ell$ slightly, say $\ell'$, we assume that $E$ and $\ell'$ intersect transversely and each intersection corresponds to a classical crossing of $D$ where $\ell$ passes exactly once.  Since the algebraic intersection number of $\ell'$ and $E$ must be zero, we see that the number of crossings of $D$ where $\ell$ passes exactly once is even.  Thus the number of crossings of $D$ where $\ell$ passes is even.
\end{proof}

\begin{defn} \label{def-good-bad}
A state of $D$ is $\it{good}$ if at each classical crossing,
the labels occur as $
\xy (-4,4);(-0.5,0.5) **@{-},
(4,-4);(0.5,-0.5) **@{-},
(4,4);(-4,-4) **@{-},
(-4,5)*{a}, (4,5)*{a}, (-4,-5)*{a}, (4,-5)*{a},
\endxy ~\hbox{or}~
\xy (-4,4);(-0.5,0.5) **@{-},
(4,-4);(0.5,-0.5) **@{-},
(4,4);(-4,-4) **@{-},
(-4,5)*{a}, (4,5)*{b}, (-4,-5)*{b}, (4,-5)*{a}
\endxy$. A state is $\it{bad}$ if at each classical crossing, the labels occur as $
\xy (-4,4);(-0.5,0.5) **@{-},
(4,-4);(0.5,-0.5) **@{-},
(4,4);(-4,-4) **@{-},
(-4,5)*{a}, (4,5)*{a}, (-4,-5)*{b}, (4,-5)*{b},
\endxy ~\hbox{or}~
\xy (-4,4);(-0.5,0.5) **@{-},
(4,-4);(0.5,-0.5) **@{-},
(4,4);(-4,-4) **@{-},
(-4,5)*{a}, (4,5)*{b}, (-4,-5)*{a}, (4,-5)*{b},
\endxy$.
\end{defn}

\begin{lem} \label{lem-2}
Let $D = D_1 \cup \dots \cup D_N$ be an oriented marked graph diagram, and let $F(D)
= F_1 \cup \dots \cup F_N$ be the proper surface in $\mathbb R^3 \times [-1,1]$ associated with $D$, where
$F_i$ $(i=1, \dots, N)$ is a connected component of $F$ and $D_i$ is the corresponding diagram.
For each $i \in \{ 1, \dots, N\}$, take a connected component of $\widehat{D}$, say $s_i$,
 such that $s_i \subset D_i$. Then,
for any $(x_1, \ldots, x_N) \in \{0,1\}^{N}$, there exists a unique bad state $\sigma$ of $D$ with $\sigma(s_i)=x_i$ $(i=1, \dots, N)$.
\end{lem}

\begin{proof}
We define a state $\sigma$ of $D$ as follows.
For each $i$,
consider a companion loop $\ell_i$ on $D$ whose base point is in $s_i$.
By Lemma~\ref{lem-1}, we can give labels $a$ and $b$ alternatively to the connected components of $\widehat{D}$ along $\ell_i$ such that $\sigma(s_i)=x_i$.  For any marked vertex on $\ell_i$, consider another companion loop whose base point is the marked vertex and give labels  $a$ and $b$ alternatively to the connected components of $\widehat{D}$ along the loop.  Continue this procedure until all connected components of $\widehat{D}$ contained in $D_i$ are given labels $a$ and $b$.  Then we see that the state $\sigma$ is a unique bad state with $\sigma(s_i)=x_i$  $(i=1, \dots, N)$.
\end{proof}

Let $D$ be a link diagram or a marked graph diagram.  For a legal state $\sigma$ of $D$, we define
$W_\sigma$ by
$$ W_\sigma = W_\sigma (x, y, z, w) = \prod_{\text{crossings}~c} \nu(c,\sigma) \in \mathbb Z[x,y,z,w].$$
Then by definition, we have
$$ [D] = [D](x,y,z,w) = \sum_{\text{legal states}~\sigma} W_\sigma \in \mathbb Z[x,y,z,w].$$

\begin{lem}\label{lem-2B}
Let $D$ be a link diagram with $sw(D)=0$.  Suppose that $D$ presents a trivial link.
If $\sigma$ is a bad  state of $D$, then $W_\sigma(0,0,z,z^{-1})=1$.
\end{lem}

\begin{proof}
Since $sw(D)=0$, we have $R^{\rm unori}_D(z) = [D](0,0,z,z^{-1})$.  Note that for a legal state $\sigma'$ that is not a bad state, $W_{\sigma'}(0,0,z,z^{-1})=0$.  Thus
$$
[D](0,0,z,z^{-1}) = \sum_{\text{bad  states}~\sigma} W_\sigma(0,0,z,z^{-1}) \in \mathbb Z[z,z^{-1}].$$
On the other hand, for a trivial link diagram $D_0$
presenting the same link with $D$, we have
$$
[D_0](0,0,z,z^{-1}) = 2^N, $$
where $N$ is the number of components of $D_0$.  Thus we have $W_\sigma(0,0,z,z^{-1}) =1$ for each bad  state $\sigma$ of $D$.
\end{proof}

\noindent
{\it Proof of (1) of Theorem~$\ref{thm:R}$.}
Let $D = D_1 \cup \dots \cup D_N$ be an oriented marked graph diagram and $F(D) = F_1 \cup \dots \cup F_N$ the proper oriented surface in $\mathbb R^3 \times [-1,1]$ associated with $D$, where $F_i$ $(i=1, \dots, N)$ is a connected component of $F$ and $D_i$ is the corresponding diagram.  Attaching minimal disks and maximal disks to $F(D)$ in $\mathbb R^4$, we obtain $\mathcal L = \mathcal L_1 \cup \dots \cup \mathcal L_N$.
In this situation, we prove that $R'_D(z)= 2^N$.

Since $R'_D(z)$ is invariant under Yoshikawa moves, applying $\Omega_1$ (and its mirror image operation, that is a consequence of $\Omega_1$ and $\Omega_2$), we may assume that $t_+(D) = sw(D_+) = 0$, where $D_+$ is the positive resolution diagram of $D$.  Then
$$
R'_D(z) = [D](0,0,z,z^{-1}) = \sum_{\text{legal states}~\sigma} W_\sigma(0,0,z,z^{-1}).$$
Since $W_\sigma(0,0,z,z^{-1})=0$ for any legal state $\sigma$ that is not a bad state, we have
$$
R'_D(z) =  \sum_{\text{bad  states}~\sigma} W_\sigma(0,0,z,z^{-1}),$$
where $\sigma$ runs over all bad states of $D$.  By Lemma~\ref{lem-2}, there exists $2^N$ bad  states.

For each bad state $\sigma$ of $D$, let $\sigma_+$ be the corresponding bad  state of $D_+$.  Then
$W_\sigma(0,0,z,z^{-1}) = W_{\sigma_+}(0,0,z,z^{-1}) = 1$ by Lemma~\ref{lem-2B}.
Thus $R'_D(z) = 2^N$.
This completes the proof of (1) of Theorem~$\ref{thm:R}$.

\begin{lem} \label{lem-3}
Let $D$ be a non-orientable marked graph diagram.  There exists a companion loop $\ell$ on $D$ such that
the number of classical crossings where $\ell$ passes is odd.
\end{lem}

\begin{proof}
It is known that, since $D$ is non-orientable, there exists a marked vertex $v$, and a companion loop $\ell$ of $D$ with base point $v$ such that the loop $\ell$ appear in a diagonal position at $v$ (cf. \cite{As}). Regard the $4$-valent graph $|D|$ as a $1$-cycle of $\mathbb R^2$ with the coefficient group $\mathbb Z/ 2 \mathbb Z$ in the sense of singular homology theory. Let $E$ be the graph in $\mathbb R^2$ obtained from $|D|$ by removing all edges where $\ell$ passes. The graph $E$ is also regarded as a $1$-cycle of $\mathbb R^2$ with the coefficient group $\mathbb Z/ 2 \mathbb Z$.  Modifying $\ell$ slightly, say $\ell'$, we assume that $E$ and $\ell'$ intersect transversely and each intersection corresponds to a classical crossing of $D$ where $\ell$ passes exactly once or the vertex $v$.  Since the algebraic ($\mathbb Z/ 2 \mathbb Z$-) intersection number of $E$ and $\ell'$  is zero, we see that the number of classical crossings of $D$ where $\ell$ passes exactly once is odd.  Thus the number of classical crossings of $D$ where $\ell$ passes is odd.
\end{proof}

\begin{lem} \label{lem-4}
Let $D$ be a non-orientable marked graph diagram.  There exist no bad states of $D$.
\end{lem}

\begin{proof}
By Lemma~\ref{lem-3}, there exists a companion loop $\ell$ of $D$ such that the number of classical crossings of $D$ where $\ell$ passes is odd.  If there exists a bad  state, the labels $0$ and $1$ appear alternatively along $\ell$.  Since the number of crossings where $\ell$ passes is odd, we have $0=1$, a contradiction.
\end{proof}

\noindent
{\it Proof of (2) of Theorem~$\ref{thm:R}$.}
Recall that
$$
[D](0,0,z,z^{-1}) = \sum_{\text{legal states}~\sigma} W_\sigma(0,0,z,z^{-1}) =
\sum_{\text{bad states}~\sigma} W_\sigma(0,0,z,z^{-1})
.$$
Since $D$ is non-orientable,
by Lemma~\ref{lem-4}, $[D](0,0,z,z^{-1})=0$.  Thus $R'_D(z)=0$.
This completes the proof of (2) of Theorem~$\ref{thm:R}$.

\begin{thm} \label{}
Let $D$ be a marked graph diagram.
Then
\begin{equation*}
S'_D=2^{N},
\end{equation*}
where $N$ is the number of connected components of the proper surface $F(D)$.
In particular, when $D$ is a marked graph diagram presenting a surface-link
with $N$ components, we have $S'_D=2^{N}$.
\end{thm}

\begin{proof}
Let $F(D) = F_1 \cup \dots \cup F_N$ be the proper surface associated with $D$, where $F_i$ $(i=1, \dots, N)$ is a connected component.  Let $D_i$ be the sub-diagram of $D$ corresponding to $F_i$.
Since $z=0, w=0$, we may consider only good states. Since the labels $0$ and $1$ do not change when we pass any classical crossing, each component $D_i$ $(i=1, \dots, N)$ has the same label.
So there exist $2^N$ good states. For each good state $\sigma$, $W_\sigma(1,1,0,0)=1$ and hence we see that $S'_D= 2^N$.
\end{proof}

Lipson's  $R^{\rm unori}_D(z)$ and $S^{\rm unori}_D(z)$ are different state  models for the same invariant of classical links.
Our invariant $R'_D(z)$ is a generalization of $R^{\rm unori}_D(z)$ to marked graph diagrams, and
$S'_D$ is a generalization of $S^{\rm unori}_D(z)$, evaluated with $z=1$, to marked graph diagrams.  As surface-link invariants, $R'_D(z)$ and $S'_D$ are different state  models for
the same invariants of orientable surface-links.  On the other hand, they are different invariants for non-orientable surface-links.


\section{Obstruction for $\Omega_4$ and $\Omega_4'$}


\begin{thm} \label{thm3}
For a marked graph diagram $D$, define $Q_D$ by
$$Q_D = (\frac{1+i}{\sqrt{2}})^{t_{+}(D)}[D](\frac{1}{\sqrt{2}},\frac{1}{\sqrt{2}},\frac{i}{\sqrt{2}},-\frac{i}{\sqrt{2}})
\in \mathbb C.$$
Then $Q_D$ is invariant under all Yoshikawa moves except $\Omega_4$ and $\Omega_4'$.
\end{thm}

\begin{proof}
From Lemma~\ref{lem1}, it has already been shown that the polynomial $[D](x,y,z,w)$ is  invariant under Yoshikawa moves $\Omega_5$, $\Omega_6$, $\Omega_6'$, and $\Omega_7$.
Since $x=\frac{1}{\sqrt{2}}$, $y=\frac{1}{\sqrt{2}}$, $z=\frac{i}{\sqrt{2}}$, and $w=-\frac{i}{\sqrt{2}}$,
the second condition (R2) of Proposition~\ref{prop1} holds, and $Q_D$ is
 invariant under Yoshikawa moves  $\Omega_2$ and $\Omega_3$.
The relations $q_1, \ldots, q_7$ also hold, and since $t_+(D)$ is preserved under $\Omega_8$, we see that $Q_D$
is invariant under $\Omega_8$.

Let $D' = \eBracket{\hcross\hcap+}$ and $D= \eBracket{\hcap+}$ be two marked graph diagrams related by a move $\Omega_1$.
Since $[D']= (x+w)[D] = \frac{1-i}{\sqrt{2}}[D]$ and $t_+(D') = 1 + t_+(D)$, we have
$Q_{D'} = \frac{1+i}{\sqrt{2}} \cdot \frac{1-i}{\sqrt{2}} Q_D = Q_D$.  Thus $Q_D$ is invariant under $\Omega_1$.
\end{proof}

\begin{exmp}\label{exmp:Omega4ind}
(1) Let $D$ and $D'$ be
admissible marked graph diagrams depicted in Fig.~\ref{Omega4ind}.
They are related by a Yoshikawa move $\Omega_4$.  By a direct calculation, we have
$[D](x,y,z,w)= 2x^2$ and $[D'] = 2x^2 + 2xz$.
Then $Q_D= 1$ and $Q_{D'}= 1+i$.

(2) Let $D$ and $D'$ be
admissible marked graph diagrams that are obtained, by changing over-under information at the classical crossings, from the diagrams
in Fig.~\ref{Omega4ind}.
They are related by a Yoshikawa move $\Omega_4'$, and
$Q_D= 1$ and $Q_{D'}= 1-i$.
\end{exmp}

\begin{figure}[h]
\begin{center}
\resizebox{0.45\textwidth}{!}{%
\includegraphics{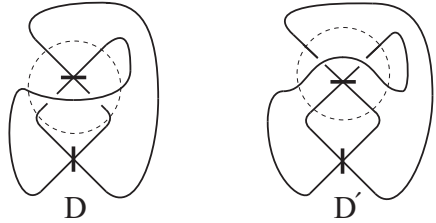} }
\caption{Two diagrams related by a Yoshikawa move $\Omega_4$}\label{Omega4ind}
\end{center}
\end{figure}

Form Theorem~\ref{thm3} and Example~\ref{exmp:Omega4ind}, we have the following.

\begin{coro}
The Yoshikawa moves $\Omega_4$ and $\Omega_4'$ are independent from the other Yoshikawa moves.
\end{coro}

{\bf Acknowlegements.}
The second author was supported by JSPS KAKENHI Grant Number 26287013.
The third author was supported by Basic Science Research Program through the National Research Foundation of Korea(NRF) funded by the Ministry of Education, Science and Technology (2013R1A1A2012446).


\begin{thebibliography}{99}

\bibitem{As} M. Asada, An unknotting sequence for surface-knots represented by ch-diagrams and their genera, \textit{Kobe J. Math.} {\bf 18} (2001), 163--180.

\bibitem{Fox}
R. H. Fox, A quick trip through knot theory, in \textit{Toplogy of $3$-manifolds and Related Topics} (Prentice-Hall, Inc., Englewood Cliffs, N.J. 1962), 120--167.

\bibitem{JKL}
Y. Joung, J. Kim and S. Y. Lee, Ideal coset invariants for surface-links in $\Bbb R\sp 4$, \textit{J. Knot Theory Ramifications} {\bf 22} (2013), No. 9, 1350052 (25 pages).

\bibitem{Ka1}
S. Kamada, Non-orientable surfaces in $\mathbb R^4$, \textit{Osaka J. Math.} {\bf 26} (1989), 367--385.

\bibitem{Ka2}
S. Kamada, \textit{Braid and Knot Theory in Dimension Four}, Mathematical Surveys and Monographs Vol. 95, (AMS, 2002).

\bibitem{Kaw}
A. Kawauchi, \textit{A survey of knot theory}, (Birkh\"auser, 1996).

\bibitem{KSS}
A. Kawauchi, T. Shibuya and S. Suzuki, Descriptions on surfaces in four-space, I; Normal forms, \textit{Math. Sem. Notes Kobe Univ.} {\bf 10} (1982), 75--125.

\bibitem{KK}
C. Kearton and V. Kurlin, All 2-dimensional links in 4-space live inside a universal 3-dimensional polyhedron, \textit{Algebr. Geom. Topol.} {\bf 8} (2008), 1223--1247.

\bibitem{KJL1}
J. Kim, Y. Joung and S. Y. Lee, On generating sets of Yoshikawa moves for marked graph diagrams of surface-links, preprint.

\bibitem{Le2}
S. Y. Lee, Invariants of surface links in $\Bbb R\sp 4$ via classical link invariants, in \textit{Intelligence of low dimensional topology 2006}, Series on Knots Everything, Vol. 40 (World Scientific Publishing, Hackensack, NJ, 2007), 189--196.

\bibitem{Le4}
S. Y. Lee, Invariants of surface links in $\Bbb R\sp 4$ via skein relation, \textit{J. Knot Theory Ramifications} {\bf 17} (2008), 439--469.

\bibitem{Le1}
S. Y. Lee, Towards invariants of surfaces in $4$-space via classical link invariants, \textit{Trans. Amer. Math. Soc.} {\bf 361} (2009), 237--265.

\bibitem{ASL}
A. S. Lipson, Some more states models for link invariants, \textit{Pacific J. Math.} {\bf 152} (1992), 337--346.

\bibitem{Lo}
S. J. Lomonaco, Jr., The homotopy groups of knots I. How to compute the algebraic $2$-type,  \textit{Pacific J. Math.} {\bf 95} (1981),
349--390.

\bibitem{So}
M. Soma, Surface-links with square-type ch-graphs, \textit{Topology Appl.} {\bf 121} (2002), 231--246.

\bibitem{Sw}
F. J. Swenton, On a calculus for $2$-knots and surfaces in $4$-space,  \textit{J. Knot Theory Ramifications} {\bf 10} (2001),
1133--1141.

\bibitem{Yo}
K. Yoshikawa, An enumeration of surfaces in four-space, \textit{Osaka J. Math.} {\bf 31} (1994), 497--522.

\end{thebibliography}
\end{document}